\documentclass[12pt,a4]{amsart}
\usepackage[a4paper, left=28mm, right=28mm, top=28mm, bottom=34mm]{geometry}
\usepackage[all]{xy}
\usepackage{amsmath}
\usepackage{amssymb}
\usepackage{amsthm}
\usepackage{ascmac}
\usepackage[alphabetic]{amsrefs}
 \usepackage{color}
 \usepackage{hyperref}
 \usepackage[capitalize]{cleveref} 
 \usepackage{wrapfig}
\usepackage[all]{xy} 
\usepackage{mathtools}
\usepackage{url}
\usepackage{mathrsfs}
\usepackage{braket} 
\usepackage[pdftex]{graphicx}
\usepackage{xcolor}
\usepackage{enumerate, enumitem, tikz}
\usetikzlibrary{intersections, calc}

\theoremstyle{definition}
\newtheorem{theo}{Theorem}[section]
\newtheorem{lem}[theo]{Lemma}
\newtheorem{cor}[theo]{Corollary}
\newtheorem{prop}[theo]{Proposition}
\newtheorem{conj}[theo]{Conjecture}

\theoremstyle{definition}
\newtheorem{ass}[theo]{Assumption}
\newtheorem{rem}[theo]{Remark}

\newtheorem{dfn}[theo]{Definition}
\newtheorem{prob}[theo]{Problem}
\newtheorem{ex}[theo]{Example}
\numberwithin{equation}{section}

\newcommand{\mb}[1]{\mathbb{#1}}
\newcommand{\bP}{\mathbb{P}} 

\newcommand{\Q}{\mathbb{Q}}
\newcommand{\cO}{\mathcal{O}}
\newcommand{\witi}{\widetilde}

\newcommand{\bealn}[1]{\begin{align*}#1\end{algin*}}
\newcommand{\maf}[1]{\mathfrak{#1}}
\newcommand{\ovl}[1]{\overline{#1}}

\DeclareMathOperator{\spec}{Spec}
\DeclareMathOperator{\proj}{Proj}
\DeclareMathOperator{\Supp}{Supp}

\title[Potential density of integral points]{On potential density of integral points on the complement of some subvarieties in the projective space}

\author{Teranishi Motoya}
\address{Department of Mathematics, Faculty of Science, Kyoto University, Kyoto 606-8502, Japan}
\email{teranishi.motoya.56s@st.kyoto-u.ac.jp}
\date{January, 17, 2024}

\begin{document}
\maketitle
\begin{abstract} 
We study some density results for integral points on the complement of a closed subvariety in the $n$-dimensional projective space over a number field. For instance, we consider a subvariety whose components consist of $n-1$ hyperplanes plus one smooth quadric hypersurface in general position, 
or four hyperplanes in general position plus a finite number of concurrent straight lines.
In these cases, under some conditions on intersection, we show that the integral points on the complements are potentially dense. 
Our results are generalizations of Corvaja--Zucconi's results for complements of subvarieties in the two or three dimensional projective space.
\end{abstract}


\section{Introduction}
One of the central problems in Diophantine Geometry is to describe the set of integral points on varieties defined over a number field $K$. Let $S$ be a set of places of $K$, containing all infinite ones. 
As in the celebrated Siegel's theorem on integral points, proving that the $S$-integral points on an affine smooth curve of genus $\geq 1$ are always finite, the abundancy for integral points has been thought to be concerned with geometric nature. 

In this paper, we are especially interested in seeking for a sufficient condition for \emph{potential density of integral points} on varieties written as complements of subvarieties in the projective space, which means that they become Zariski dense after a finite extension of the ground field and the set $S$. 

Several conjectures and results for special cases of them are provided on the problems of potential density of integral points. In this paper, we are interested in the following two conjectures. 
\begin{conj}[\cite{HT01}] \label{conj:anticanonical}
Let $(X,D)$ be a pair with $X$ a smooth projective variety and $D$ a reduced effective anticanonical divisor with at most normal crossings singularities. Are the integral points on $X\setminus{D}$ potentially dense?
\end{conj}
\begin{conj}[{The Puncturing Conjecture, \cite{HT01}, \cite{CoZu23}}] \label{conj:puncturing}
Let $X$ be a smooth (quasi-)projective variety over a number field $K$ and $Z$ a subvariety of codimension $\geq 2$. Assume that the rational points on $X$ are potentially dense. Are the integral points on $X\setminus{Z}$ are potentially dense?    
\end{conj}

Much of results are known for 2 or 3-dimension. For example,  \cref{conj:anticanonical} is proved for $X$ smooth del Pezzo surfaces (see \cite{Coc23}, \cite{HT01}),  for $X$ elliptic K3 surfaces (see \cite{BT00}, \cite{LaNa22}) and for some Fano 3-fold or complements in $\bP^{3}$ \cite{CoZu23}. 
\cref{conj:puncturing} is proved for varieties in some class containing toric varieties \cite{HT01}, for some Fano 3-fold \cite{DR22}. Besides, there is a work by Levin and Yasufuku \cite{LY16} which studies potential density of an affine surface $X$ given as the complement in a curve in $\bP^{2}$ via its logarithmic kodaira dimension $\overline{\kappa}(X)$. Note that \cref{conj:anticanonical} is contained in the case $\overline{\kappa}(X) = 0$. 

Our interest is to study these two conjectures for $X=\mathbb{P}^{n}_K$, and we proved the following theorem. 
\begin{theo} \label{MainThm1_n-1planes_plus_1quadric}
 Let $n \geq 2$, and $D$ be a divisor on $\bP_K^{n}$ of the following form:
 \[D = H_1 + \dots + H_{n-1} + Q.\]
 Here,  $H_1,\dots, H_{n-1} \subset \bP_K^{n}$ are hyperplanes over a number field $K$ in general position and $Q$ is a smooth quadric hypersurface over $K$. Suppose that the line $L\coloneqq H_1\cap \dots \cap H_{n-1}$ and $Q$ have two $K$-rational intersection point.
 Then the integral points on $\bP_K^n\setminus{D}$ are potentially dense.
\end{theo}

   \begin{theo} \label{MainThm3_concurline_puncturing}
    Let $D\subset \bP^{3}_K$ be a closed subvariety of the following form:
    \[
    D = H_1 + H_2 + H_3 + H_4 + L_1 + \dots + L_r
    \]
    Here, $H_1,\dots, H_4$ are hyperplanes over $K$ in general position, and $L_1,\dots, L_r$ are concurrent lines over $K$ passing through a common $K$-ratinal point $p$. Suppose also that each of $L_1,\dots, L_r$ does not intersect the 6 lines $\bigcup_{1\leq i<j\leq 4}H_i\cap H_j$. Then the integral points on $\bP^{3}_K\setminus{D}$ are potentially dense. 
    \end{theo}

    \begin{figure}[!] 
    \begin{tikzpicture}
    \draw (-2,-2)--(2,2);
    \fill (1,1) circle(2pt);
    \filldraw[ball color=black!5!white, opacity=0.8] (0,0) circle(1.9);
    \draw (-2,-2)--(-1, -1);
    \fill (-1,-1) circle(2pt) node[right]{$p$};
    \draw node at (-1.2,1.2) [above left]{$Q$};
    \filldraw[color=black, opacity=0.1] (-4,-2)--(1,-2)--(5,2)--(0,2)--cycle;
    \draw node at (2,3) [right]{$H_1$};
    \filldraw[color=black, opacity=0.1] (2,3.6)--(-2,-0.4)--(-2,-3)--(2,1)--cycle;
    \draw node at (4,2) [above]{$H_2$};
    \end{tikzpicture}
    \caption{$D=H_1+H_2+Q$ in \cref{MainThm1_n-1planes_plus_1quadric} ($n=3$)}
    \end{figure}
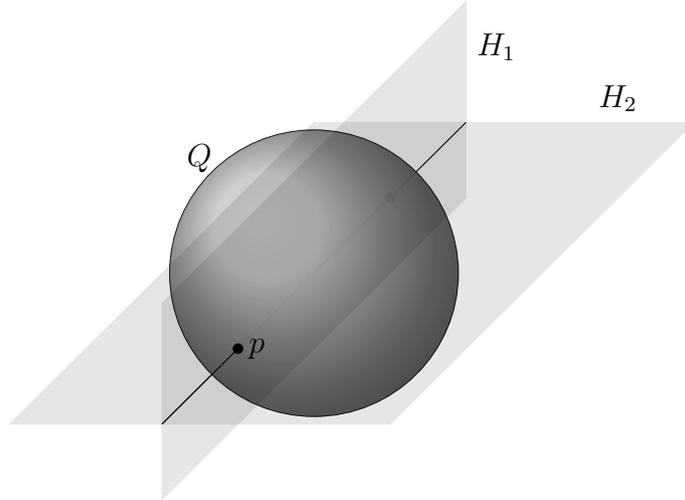

\begin{figure}[!] 
\begin{tikzpicture}
\draw[dashed] (-2,-2)--(2,2);%
\draw[dashed] (-3,-1)--(2,-1);%
\draw[dashed] (-1,1.55)--(-1,-2.5);
\fill (-2.5,0) circle(2pt) node[below]{$p$};
\draw (0.5,1)--(-5.5,-1) node at (-5.5,-1.4) {$L_1$};
\draw (-4.5,1)--(1.4,-1.9) node at (-4.3,1) [above]{$L_2$};
\filldraw[color=black, opacity=0.1] (2.3,4.9)--(-2,0.6)--(-2,-3.5)--(2.3,0.8)--cycle; %
\draw node at (2,2.3) [right]{$H_1$};
\filldraw[color=black, opacity=0.1] (-4,-2)--(1,-2)--(5,2)--(0,2)--cycle;
\draw node at (4,2) [above]{$H_2$};
\filldraw[color=black, opacity=0.1] (-3,-2.5)--(2,-2.5)--(2,1.55)--(-3,1.55)--cycle;
\draw node at (1.5,-3) [right]{$H_3$};
\filldraw[color=black, opacity=0.1] (-3,2)--(-2,3.5)--(4,5)--(3,3.5)--cycle;
\draw node at (-2.8,3) {$H_4$};
\end{tikzpicture}
\caption{$D=H_1+H_2+H_3+H_4+L_1+L_2$ in \cref{MainThm3_concurline_puncturing}.}
\end{figure}
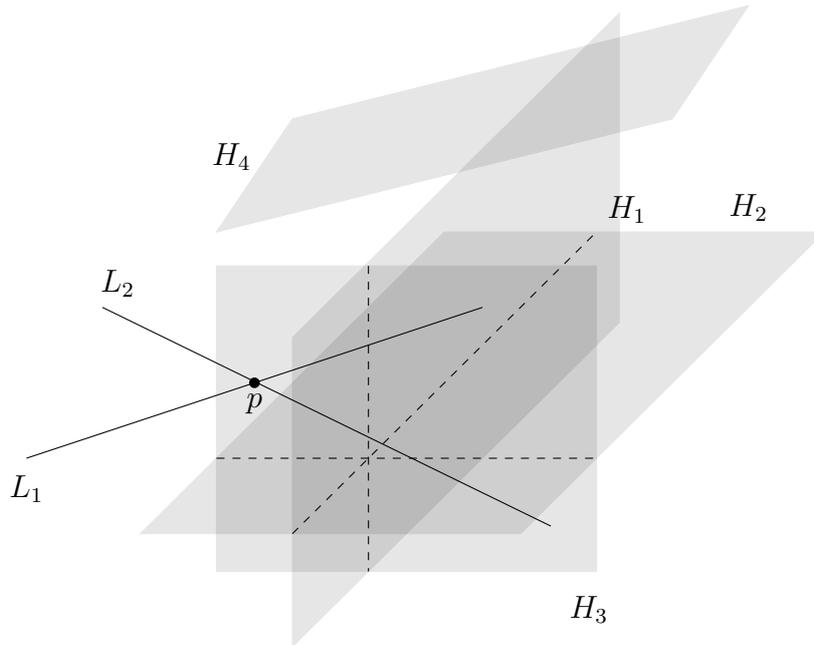 

Theorem \ref{MainThm1_n-1planes_plus_1quadric}
is a higher dimensional generalization of a result obtained by Corvaja-Zucconi \cite[Theorem 3.3.2]{CoZu23}.
The strategy of our proof of \cref{MainThm1_n-1planes_plus_1quadric} 
is similar to Corvaja--Zucconi's proof. The key step is to find sufficiently many straight lines intersecting $D$ in two coprime points and to apply a lemma of Beukers \cite{Beu95} (see \cref{prop:Beukers_Lemma}) to search for infinitely many integral points lying on a straight line.

\cref{MainThm3_concurline_puncturing} proves the puncturing conjecture for $X = \bP^{3}_K$ and $Z$ a finite number of concurrent straight lines with some condition on their intersection, but note that this result is true by the case of toric varieties in \cite{HT01}. The motivation behind this theorem is to give a similar proof of Corvaja-Zucconi's result \cite[Lemma 3.2.1]{CoZu23}, which 

we directly construct sufficiently many $S$-integral points by using appropriate $S$-units for their coordinates. The condition on the intersection of $L_1,\dots, L_r$ and $\bigcup_{1\leq i<j \leq 4}H_i \cap H_j$ is necessary for our constructions of integral points.  


The organization of this paper is as follows. In \cref{section:preliminary}, we review some basic about integral points on varieties. In \cref{section:Beukers_Lemma}, we review a lemma of Beukers on integral points on straight lines. In \cref{section:main_result}, we prove the main theorems of this paper and mention its generalization.



\section{Preliminaries} \label{section:preliminary}

This section is included to recall some basic definitions such as integral points on varieties. Our main reference is \cite{Cor16}. 
 
\subsection{Notation} 

\begin{itemize}
\item $K$ is a number field, and $\cO_K$ is the ring of integers. 
\item $M_K$ is the set of all places of $K$, $M_{K}^{{\rm fin}}$ is the set of all finite ones, and $M_K^{\infty}$ is the set of all infinite ones.
\item $S$ is a finite set of places of $K$, and we always suppose that $S$ contains the infinite ones.  
\item $\cO_S \coloneqq \set{x\in K | \, v(x)\geq 0 \, \, \text{for all $v\notin S$}}$ is the ring of $S$-integers. 
\item  $\cO_v \coloneqq \set{x\in K | v(x)\geq 0}$ is the valuation ring at a finite place $v$ with its maximal ideal $\mathfrak{m}_{v} \coloneqq \set{x\in K \mid \,\,  v(x)>0}$. Let $k_v \coloneqq \cO_v / \mathfrak{m}_v$ be the residue field of $\cO_v$. 
\item By a \emph{variety} we shall mean a separated and finite type scheme over a field $k$. The set of $k$-rational points is denoted by $X(k)$. For quasi-projective varieties, we often specify a closed immersion to the projective space. 
\item For a homogeneous ideal $I$ of the polynomial ring $k[X_0,\dots, X_n]$ over a field $k$, the zero set of $I$ in $\bP^{n}_{k}$ is denoted by $V_{+}(I)$. 
\end{itemize}

\subsection{Reduction of a subvariety}
We recall the notion of the reduction of a subvariety on the projective space $\bP_K^n$ over a number field $K$. This is necessary for our definition of integral points.

\begin{dfn}[{\cite[Section 1.1]{Cor16}}] \label{dfn:reduction}
Let $Z\subset \bP_{K}^n$ be a subvariety over $K$ defined by the radical homogeneous ideal $I$ of $K[X_0,\dots, X_n]$. Let $v\in M_{K}^{\rm fin}$. The \emph{reduction of $Z$ at $v$} is the subvariety $Z_v\subset \bP^{n}_{k_v}$ defined by the ideal obtained by the canonical image of $I_v \coloneqq I\cap \cO_v[X_0,\dots X_n]$ in $k_v[X_0, \dots ,X_n]$.
\end{dfn}

\begin{dfn} \label{dfn:reduce_to_Z}
We denote by $x_v$ the reduction of a point $x\in \bP_{K}^{n}(K)$ at $v\in M_K$. When $x_v\in Z_v$, we say that \emph{$x\in \bP_{K}^n(K)$ reduces to $Z$ at $v$.} 
\end{dfn}

\begin{ex}
Let $I = (XY - 4Z^2) \subset \mathbb{Q}[X,Y,Z]$ be an ideal. The curve $V_{+}(I)$ is an irreducible conic over $\Q$. The reduction of $V_{+}(I)\subset \bP^2_{\mathbb{Q}}$ at the 2-adic valuation $v\in M_{\Q}^{\rm fin}$ is 
\[V_{+}(I_vk_v[X,Y,Z]) = V_+((XY)\mathbb{F}_{2}[X,Y,Z]) \subset \bP^{2}_{\mathbb{F}_{2}}\]
This is a reducible curve whose components are two lines over $\mb{F}_{2}$. 
\end{ex}

\begin{rem} \label{rem:reduction_in_scheme}  
In the language of scheme theory, $Z_v$ is naturally constructed as follows. Let $\bP^{n}_{\cO_K} = \proj{\cO_K[X_0,\dots, X_n]}$ be the projective space over $\cO_K$. We have an embedding of $Z$ in the generic fiber $\bP^{n}_{K}$, and we construct $Z_v$ as the special fiber at $\mathfrak{p}_v\in \spec{\cO_K}$ (where $\mathfrak{p}_v\cO_v = \mathfrak{m}_v$) of the Zariski closure $\mathcal{Z}$ of $Z$ in $\bP^{n}_{\cO_K}$. Indeed, if $Z = V_{+}(I)$ where $I$ is a homogeneous ideal of $K[X_0,\dots, X_n]$, then $\mathcal{Z} = V_+(I\cap \cO_K[X_0,\dots, X_n])$ and $Z_{v}$ is the closed fiber of $\mathcal{Z}$ at $v$, i.e., 
\[\mathcal{Z}\times_{\cO_K}\spec{k_v} = \proj{\left(k_v[X_0,\dots, X_n]/I_v k_v[X_0,\dots, X_n]\right)}.\] 
\end{rem}
For two closed subvarieties $Z, W\subset \bP_{K}^n$, taking their (scheme theoretic) intersection does not necessarily commute with taking their reduction. In other words, there is a possibility that we have $(Z\cap W)_v \subsetneq  Z_{v} \cap W_{v}$ for some place $v\in M_K^{\text{fin}}$ because we may have
\begin{multline*} 
(\mathcal{I}(Z)+\mathcal{I}(W))\cap  \cO_v[X_0,\dots ,X_n] \\
\supsetneq 
(\mathcal{I}(Z)\cap \cO_v[X_0,\dots ,X_n]) + (\mathcal{I}(W)\cap \cO_v[X_0,\dots ,X_n])
\end{multline*}
where $\mathcal{I}(Z), \mathcal{I}(W) \subset K[X_0,\dots, X_n]$ are the homogeneous ideals of $Z, W$, respectively. However, we may see that this occurs at only a finite number of places. 



\begin{lem} \label{lem:bad-reduction-intersection}
Let $Z, W\subset \bP^{n}_K$ be a closed subset over $K$. Then the set 
\[
S = \set{v\in M_K^{\text{fin}} \mid (Z\cap W)_v \subsetneq Z_{v} \cap W_{v}} \cup M_{K}^{\infty}
\]
is finite. 
\end{lem}

\begin{proof}
Let $Z = V_+(I)$, $W = V_+(J)$ where $I,J\subset K[X_0,\dots, X_n]$ are homogeneous ideals. Consider the quotient $\cO_K[X_0,\dots, X_n]$-module 
\[
M\coloneqq \dfrac{(I + J)\cap \cO_K[X_0,\dots, X_n]}{(I\cap \cO_K[X_0,\dots, X_n]) + (J \cap \cO_K[X_0,\dots, X_n])}. 
\]
Let $T\coloneqq \cO_K\setminus{\set{0}}$ be a multiplicatively closed set of $\cO_K[X_0,\dots, X_n]$. Then we have 
\begin{align*}
T^{-1}M &= M\otimes_{\cO_K[X_0,\dots, X_n]} K[X_0,\dots,X_n] \\
&\cong \dfrac{T^{-1}((I+J)\cap \cO_K[X_0,\dots, X_n])}{T^{-1}(I\cap \cO_K[X_0,\dots, X_n]) + T^{-1}(J\cap \cO_K[X_0,\dots, X_n])} \\
&= (I+J)/(I+J) \\ 
&= 0.
\end{align*}
Here, the third equality follows from the fact that an ideal of $T^{-1}\cO_{K}[X_0,\dots, X_n] = K[X_0,\dots, X_n]$ is the extension of an ideal of $\cO_K[X_0,\dots, X_n]$ (see \cite[Proposition 3.11 (i)]{AM69}). 
Since $\cO_K[X_0,\dots, X_n]$ is a noetherian ring, the module $M$ is finitely generated over $\cO_K[X_0,\dots, X_n]$. Hence we may find $t\in T$ such that $tM = 0$. Now, let us define the set $S'$ as 
\[
S'\coloneqq M_{K}^{\infty} \cup \set{ v\in M_{K}^{\text{fin}} \mid v(t) > 0},
\]
and let $I_v, J_v$ and $(I+J)_v$ are ideals of $\cO_v[X_0,\dots, X_n]$ as in \cref{dfn:reduction}. Then, $t$ is a unit in $\cO_v$ for any $v\notin S'$, and hence we have 
\begin{align*}
&\quad M\otimes_{\cO_K[X_0,\dots, X_n]} \cO_{v}[X_0,\dots, X_n]\\ 
&=  \dfrac{(I + J)_v\cap \cO_K[X_0,\dots, X_n]}{(I_v\cap \cO_{K}[X_0,\dots, X_n]) + (J_v \cap \cO_{K}[X_0,\dots, X_n])}\otimes_{\cO_{K}[X_0,\dots, X_n]} \cO_v[X_0,\dots, X_n] \\ 
&= \dfrac{(I+J)_v}{I_v + J_v}\\
&= 0.
\end{align*}
Therefore, it follows that the set $S$ is contained in the finite set $S'$.
\end{proof}


\subsection{Potential density of integral points}
Let us review some basic results of integral points on varieties and the notion of its potential density. 
    \begin{dfn}
    Let $X\subset \bP_{K}^{n}$ be a quasi-projective variety, let $D\subset X$ be a proper closed subvariety over $K$, and let $\witi{X}\subset \bP^{n}_K$ be the Zariski closure of $X$.
     We say that a $K$-rational point $x\in X(K)$ is an \emph{$S$-integral point on $X\setminus{D}$} if $x$ does not reduce to $D\cup (\witi{X}\setminus{X})$ at all finite places outside $S$. We write the set of $S$-integral points on $X\setminus{D}$ by $(X\setminus{D})(\cO_S)$. 
     If $D = \set{y}$ is a single point and if $x$ is an $S$-integral point on $X\setminus{\set{y}}$, we say that \emph{$x$ and $y$ are $S$-coprime.}
    \end{dfn}
We also include the definition of integral points on quasi-projective varieties. Note that $K$-rational points are exactly $M_K^{\infty}$-integral points with $D$ empty. When $D$ is a divisor, it is called the \emph{divisor at infinity}. 

\begin{rem}
For different definitions of integral points, see \cite{CoZa18} for example. 
\end{rem}

As with rational points, we may think of the quantitative study of interal points such as its density or degeneracy. We are interested in the density of integral points after an enlargement of $K$ and $S$. 
\begin{dfn} 
Let $X\subset \bP^n_{K}$ be a projective variety over $K$, and $D\subset X$ be a proper closed subvariety over $K$. We say that \emph{the integral points on
$X\setminus{D}$ are potentially dense} if there exists a finite extension $K'$ of $K$ and a finite set $S'$ containing all the places lying over those of $S$ such that $(X\setminus{D})(\cO_{S'})$ is Zariski dense in $X(K')$.  
\end{dfn}
For integral points on an affine variety, in other words if $D$ is the ``hyperplane at infinity'' of the projective space, we may see a naive definition of integral points. 
Namely, they are the points with $S$-integer coordinates.

\begin{prop} \label{prop:affine_intpt}
Let $X\subset \bP^{n}_K$ be a variety over $K$, and
let $D$ be a divisor given by  
\[D = \set{[X_0:\dots:X_N] \in X(K) \mid X_N = 0}.\]
For a $K$-rational point 
\[x \coloneqq [x_0:\dots: x_{N-1}:1] \in (X\setminus{D})(K),\]
the following are equivalent. 
\begin{enumerate}[label=(\roman*)]
    \item The point $x$ is an $S$-integral point on $X\setminus{D}$.  
    \item $x_i \in \cO_S$ for all $i\in \set{0,1,\dots, N-1}$. 
\end{enumerate}
\end{prop} 
\begin{proof} Let $\pi_v$ be a uniformalizer of $\mathfrak{m}_v$, and let $x_N = 1$.  
Suppose (i), and suppose that $x_i \notin \cO_S$ for some $i$. Then for some places $v\notin S$, the minimum of valuations $e \coloneqq \displaystyle \min_{0\leq i\leq N}v(x_i)$ is smaller than $0$. So we have 
\begin{align*}
    x_v &= [x_0\pi_{v}^{-e}:\dots: x_{N-1}\pi_{v}^{-e}:\pi_{v}^{-e}]_v \\
    &= [x_0\pi_{v}^{-e}:\dots: x_{N-1}\pi_{v}^{-e}:0]_v,
\end{align*}

and hence $x$ reduces to $D$ at $v$, a contradiction.
Conversely, suppose (ii). Then for all places $v$ of $K$ outside $S$, we have $e = v(1) = 0$. Therefore, the reduction $x_v$ is exactly 
\[x_v = [x_0 \bmod{\maf{m}_v} : \dots : 1\bmod{\maf{m}_v}] \notin D_{v},\]
and hence $x$ is an $S$-integral point on $X\setminus{D}$. 
\end{proof}
In a similar way, we can prove the following: 
    \begin{prop}
    Let $X\subset \bP^{n}_K$ be a variety over $K$, and
    let $D$ be a divisor given by  
    \[D = \set{[X_0:\dots:X_N] \in X(K) \mid X_0X_1\dots X_{N-1}X_N = 0}.\]
    For a $K$-rational point 
    \[x \coloneqq [x_0:\dots: x_{N-1}:1] \in (X\setminus{D})(K),\]
    the following are equivalent. 
    \begin{enumerate}[label=(\roman*)]
        \item $x$ is an $S$-integral point on $X\setminus{D}$.  
        \item $x_i \in \cO_S^{\ast}$ for all $i\in \set{0,1,\dots, N-1}$. 
    \end{enumerate}
    In particular, if $\cO_S^{\ast}$ is of infinite group, the set $(\bP^{n}_K\setminus{D})(\cO_S)$ is Zariski dense in $\bP^{n}_K$. 
    \end{prop}

The property that two $K$-rational points are $S$-coprime can be described in an algebraic way. 
\begin{prop} \label{prop:Coprime-vs-ideals}
Let $x = [x_0:\dots :x_n]$ and $y=[y_0:\dots: y_n]$ be $K$-rational points in $\bP_{K}^{n}$. Then the following is equivalent. 
\begin{enumerate}
\item $x$ and $y$ are $S$-coprime. 
\item For any $v\notin S$, the following equation of fractional ideals of $\cO_v$ holds. 
\begin{equation} \label{eq:eq-ideal} 
\sum_{i,j} (x_{i}y_j - x_jy_i)\cO_v = (x_0,x_1,\dots, x_{n})(y_0, y_1,\dots, y_n)\cO_v
\end{equation}
\end{enumerate}
\end{prop}
\begin{proof} 
Let $\pi$ be a uniformalizer of the maximal ideal $\maf{m}_{v}$, and let 
\[
m_1 = \min{\set{v(x_0),\dots,v(x_n)}},\quad m_2 = \min{\set{v(y_0),\dots,v(y_n)}}.
\] 
Suppose (1), then 
\[
[\ovl{x_0\pi^{-m_1}}:\ovl{x_1\pi^{-m_1}}:\dots :\ovl{x_n\pi^{-m_1}}] \neq [\ovl{y_0\pi^{-m_2}}:\ovl{y_1\pi^{-m_2}}:\dots:\ovl{y_n\pi^{-m_2}}]. 
\]
Here, $\ovl{a}$ denotes the canonical class of $a\in \cO_S$ in the residue field $k_v$. So, there exists some $i',j'$ such that 
\begin{align*}
(\ovl{x_{i'}\pi^{-m_2}})\cdot (\ovl{y_{j'}\pi^{-m_1}}) -  (\ovl{x_{j'}\pi^{-m_2}})\cdot (\ovl{y_{i'}\pi^{-m_1}}) \neq 0
\end{align*} 
in $k_v$. This is equivalent to $v(x_{i'}y_{j'} - x_{j'}y_{i'}) = m_1 + m_2$. For other $i,j$ we have $v(x_iy_j - x_jy_i) \geq m_1 + m_2$ by the definition of $m_1$ and $m_2$. Thus (2) holds. The converse is also true by the same argument.
\end{proof}
To handle integral points or their reduction in a geometric way, it is more convenient to use the definition of integral points as $\cO_S$-sections of arithmetic varieties. 

\begin{prop} 
Let $\witi{X}\subset \bP^{n}_{K}$ be a projective variety, and let $D\subset \witi{X}$ be a closed subvariety, both defined over $K$. 
Let $\mathcal{X}\subset \bP^{n}_{\cO_{S}}$ be the closure of generic fiber $\witi{X}$. Let $\mathcal{D}$ be the Zariski closure of $D$. Then we have a bijection
\[
\set{\text{$S$-integral points on $\witi{X}\setminus{D}$}} \xrightarrow{1:1} \set{\text{sections $\spec{\cO_S}\to \mathcal{X}\setminus{\mathcal{D}}$ }}.
\]
\end{prop}
\begin{proof}
When we have a rational point $x$ of $\witi{X}$, we obtain a morphism 
\[\spec{K} \to \witi{X} = \mathcal{X}\times_{\cO_S}\spec{K} \to \mathcal{X}.\]
The structure map $\mathcal{X}\to \spec{\cO_S}$ is proper because it factors as $\mathcal{X} \to \bP^{n}_{\cO_S} \to  \spec{\cO_{S}}$. Hence, we obtain a unique morphism $\spec{\cO_v} \to \witi{X}$ commuting the following diagram by the valuative criterion of properness (\cite[Corollary II.4.8]{Har}).  
\[
\xymatrix{
\spec{K} \ar[d]\ar[r]^{x} & \mathcal{X} \ar[d] \\ 
\spec{\cO_v} \ar[r] \ar@.[ur] & \spec{\cO_S} 
}
\]
The morphism $\spec{\cO_v} \to \mathcal{X}$ is extended to some open set of $\spec{\cO_S}$ (\cite[Exercise 3.2.4]{Liu02}). So, we obtain a morphism $\spec{\cO_S} \to \mathcal{X}$ by gluing extensions of $\spec{\cO_v}\to \mathcal{X}$ for all $v\notin S$ (They can be glued by uniqueness). 

For now, let $x \in (\witi{X}\setminus{D})(\cO_S)$. Denote also $x$ as the morphism $\spec{\cO_S}\to \mathcal{X}$ we have obtained. The reduction $x_v$ at $v\notin S$ is the specialization 
\[\spec{k_v} \to \spec{\cO_S}\to \mathcal{X}_{v}.\]
By assumption, the image point of $x_v:\spec{\cO_S} \to \mathcal{X}_v$ is not contained in $\mathcal{D}_v$ (see \cref{rem:reduction_in_scheme}).
Therefore, $x$ induces a section $\Phi(x): \spec{\cO_S} \to \mathcal{X} \setminus{\mathcal{D}}$. 

Conversely, given a section $s:\spec{\cO_S}\to \mathcal{X}\setminus{\mathcal{D}}$,  we naturally obtain a rational point $\Psi(s): \spec{K} \to X$ by considering following composition:
\begin{multline*}
\Psi(s): \spec{K} \xrightarrow{\text{can.}}\spec{\cO_S}\times_{\cO_S}\spec{K} \\
\longrightarrow (\mathcal{X}\setminus{\mathcal{D})}\times_{\cO_S}\spec{K} \xrightarrow{\text{imm.}} \mathcal{X}\times_{\cO_S}\spec{K} \xrightarrow{\cong} X.
\end{multline*}
It is easy to show that $\Psi(\Phi(x)) = x$ and $\Phi(\Psi(s)) = s$.  
\end{proof}

We conclude this section with referring to facts about potential density of integral points. The following gives a bound on the number of components to have a potentially dense set of integral points.
    \begin{prop}[{\cite[Theorem 1.2]{NoWin02}, \cite[Section 5.4]{Cor16}}] \label{NoWin02degeneracy} 
    Let $\witi{X}$ be a smooth projective variety over $K$. Let $q(\witi{X})$ be its irregularity and $\rho(\witi{X})$ be the rank of its N\'{e}ron-Severi group. Let $D_1,\dots, D_{l}$ be hypersurfaces of $\witi{X}$ in general position. If $l > \dim{\witi{X}} + \rho(\witi{X}) - q(\witi{X})$, 
    then the set $(\witi{X}\setminus{(D_1\cup \dots \cup D_l)})(\cO_S)$ is not Zariski dense for any ring of $S$-integers $\cO_S$. 
    \end{prop}
\begin{ex}
Let $\witi{X} = \bP^{n}_K$ and $D$ a divisor with $l$ irreducible component in general position. If $l > n+1$, then the integral points on $\bP^{n}_{K}\setminus{D}$ is \emph{not} potentially dense.  
\end{ex}

Our \cref{MainThm3_concurline_puncturing} is a partial generalization of the following. 
\begin{prop}[{\cite[Lemma 3.1.1]{CoZu23}}] \label{threelines_points}
Let $Y\subset \bP^2_{K}$ be the closed subvariety formed by the union of three lines in general position and a finite set of points outside the three lines. Then the integral points on $\bP^2_{K}\setminus{Y}$ are potentially dense.
\end{prop}

\subsection{Behaviour of integral points under morphisms} \label{Behaviour_under_mor}
The goal of this section is to illustrate that having a potentially dense set of integral points is an isomorphic invariant of quasi-projective varieties.

It is possible for two isomorphic quasi-projective varieties over $K$ to have different numbers of integer points.

\begin{ex} 
Consider the two plane conics
\begin{align*}
C_1 &\coloneqq \set{[X_0:X_1:X_2] \mid X_0X_1=X_{2}^2},\\
C_2 &\coloneqq \set{[X_0:X_1:X_2]\mid X_0X_1 = 2X_{2}^2},
\end{align*} 
and let $D$ be the line $V_+(X_2)$ in $\bP^{2}_K$. Although the two curves $C_1$ and $C_2$ are isomorphic over $\mathbb{Q}$ and there is an isomorphism $[X_0:X_1:X_2]\mapsto [2X_0:X_1:X_2]$ fixing $D$, the number of integral points ($\mathbb{Z}$-points) of them are different. Indeed, we have 
\begin{align*}
(C_1\setminus{D})(\mathbb{Z}) &= \set{[1:1:1], [-1:-1:1]} \\ 
(C_2\setminus{D})(\mathbb{Z}) &= \set{[1:2:1], [2:1:1], [-1:-2:1], [-2:-1:1]}.
\end{align*}
However, if the set $S\subset M_{\Q}$ contains the 2-adic valuation, we obtain a bijection of $S$-integral points.  
\end{ex}

As can be seen from the example above, when we study integral points by sending to other quasi-projective varieties with a morphism over $K$, we have to care about the coefficients of its local representation, because it is not necessarily extended to a morphism over $\cO_K$. However, we may see that the morphism is extended to a morphism over $\cO_S$ for some finite set $S\subset M_{K}$ containing all infinite places, and hence $S$-integral points on the domain should be sent to those on the target. Namely, 
    \begin{prop}[{\cite[Section 1.3]{Cor16}}] \label{prop:sends_to_intpt} 
    Let $\witi{X}_1 \subset \bP_{K}^{n}$ and $\witi{X}_2 \subset \bP_{K}^{m}$ be quasi-projective varieties, and let $D_1 \subset \witi{X}_1$ and $D_2\subset \witi{X}_2$ be divisors, all defined over a number field $K$. Let $\pi:\witi{X}_1\to \witi{X}_2$ be a $K$-morphism such that $\pi^{-1}(D_2) = D_1$. Enlarging $S$ if necessary, $\pi$ sends $S$-integral points on $\witi{X}_1\setminus{D_1}$ to $S$-integral points on $\witi{X}_2\setminus{D_2}$, i.e., we have
    \[\pi((\witi{X}_1\setminus{D_1})(\cO_S)) \subset (\witi{X}_2\setminus{D_2})(\cO_S).\]
    \end{prop}
\begin{proof}
Consider an affine open subset $U\subset X_1\setminus{D_1}$ and $V\subset X_2\setminus{D_2}$ such that $\pi(U)\subset V$. Each $U, V$ are isomorphic to a quasi-projective variety embedded in $\mb{A}^n$ and $\mb{A}^{m}$, respectively. Then the restriction $\pi|_{U}:U\to V$ is given by 
\[
\pi(x_1,\dots, x_n) = (f_1(x_1,\dots, x_n), \dots, f_k(x_1,\dots, x_n))
\]
for some polynomials $f_1,\dots, f_k\in K[X_1,\dots, X_n]$. By enlarging $S$, we can say that $f_1,\dots, f_k\in \cO_S[X_1,\dots, X_n]$. Note that for $(x_1,\dots, x_n)\in U(\cO_S)$ we have $x_i \in \cO_S$ (\cref{prop:affine_intpt}), and so we have $f_i(x_1,\dots, x_n)\in \cO_S$ for all $i$. 

Since $\pi$ is a glueing of finite number of restriction of the form $U\to V$, we obtain an $\cO_S$-morphism $\witi{\mathcal{X}}_1\setminus{\mathcal{D}_1}\to \witi{\mathcal{X}}_2\setminus{\mathcal{D}_2}$ between $\cO_S$-models, where $\witi{\mathcal{X}_1}$ and $\witi{\mathcal{X}_2}$ are the closure of $\witi{X}_1$ and $\witi{X}_2$ in $\bP^{n}_{\cO_S}$. Therefore, $\pi$ sends an $\cO_S$-section $s:\spec{\cO_S}\to \witi{\mathcal{X}}_1\setminus{\mathcal{D}_1}$ to $\pi\circ s$, an $\cO_S$-section $\spec{\cO_S}\to \mathcal{X}_2\setminus{\mathcal{D}_2}$. 
\end{proof}

\begin{cor} \label{cor:not_pd_heredity}
In the setting of \cref{prop:sends_to_intpt}, suppose also that $\pi:\witi{X}_1\to \witi{X}_2$ is dominant. If the integral points on $\witi{X_1}\setminus{D_1}$ are potentially dense, then those on $\witi{X}_2 \setminus{D_2}$ are also potentially dense. In particular, if $\pi$ is an isomorphism over $K$, the converse is also true. 
\end{cor} 
\begin{proof}
For any continuous map $f:X\to Y$ between topological spaces, the image of a dense subset in $X$ is also dense in the image $f(X)$. Hence the claim holds by dominancy of $\pi$ and \cref{prop:sends_to_intpt}. 
\end{proof}

\begin{ex}[The case of $\bP_{K}^1$] \label{ex:P^1-case}
Let $n\geq 1$ be an integer. 
Let $D\subset \bP_{K}^1$ be a divisor consisting of $n$ distinct $K$-rational points. Then the integral points on $\bP_{K}^1\setminus{D}$ are potentially dense if and only if $n=1,2$. Indeed, when $n\geq 3$, we can choose coordinates in $\bP^1_K$ so that $D$ contains the three points $[0:1]$, $[1:1]$, $[1:0]$. Then if $[x:1] \in \bP_{K}^1\setminus{D}$ is an integral point, we have $x\in \cO_S^{\ast}$ and $x-1 \in \cO_S^{\ast}$. So the quantity of integral points concerns with a unit equation $U+V = 1$, where $U,V\in \cO_S^{\ast}$. It is known that this equation has only finitely many of solutions for any $S$ and $K$ (see \cite[Theorem D.8.1]{HiSi00}). Hence the set $(\bP_{K}^1\setminus{D})(\cO_S)$ is finite. 
\end{ex}

\begin{ex}[The case of $(\bP_{K}^2, \text{concurrent 3 rational lines})$]
In \cref{conj:anticanonical}, the singularity of a divisor $D$ should be with at most normal crossings. Without this condition, the conjecture becomes false. For example, let $X = \bP_{K}^2$ and let $D$ be a divisor consisting of concurrent three lines over $K$. This is the exceptional case of \cite[Theorem 1.2]{Coc23}. The integral points on $X\setminus{D}$ are actually \emph{not} potentially dense. Indeed, drawing lines passing through the common point of the three induces a morphism 
\[
X\setminus{D}\to \bP_{K}^1\setminus{\set{P_1, P_2, P_3 }}
\]
where $P_1,P_2,P_3$ are distinct $K$-rational points. After enlarging $S$, this morphism sends integral points to those on the target, but the set of integral points on the target is finite by \cref{ex:P^1-case}. Therefore, the set $(X\setminus{D})(\cO_S)$ is always contained in a finite number of straight lines for any $S$ and $K$.
\end{ex}

\begin{ex}
If an affine variety $X\subset \mathbb{A}^{n}_{K}$ over $K$ admits a dominant morphism $X\to C$, where $C\subset \mathbb{A}^{n}_{K}$ is a smooth affine curve of genus $\geq 1$, then the integral points on $X$ is not potentially dense by Siegel's theorem on integral points. 
\end{ex}

\section{Beukers' Lemma} \label{section:Beukers_Lemma}
\subsection{Integral points on a line}
In this section, we assume that \emph{$\cO_S$ is a principal ideal domain.} Note that $\cO_S$ becomes a principal ideal domain if we enlarge $S$ so that $S$ contains all prime divisors appearing in a complete representative system for the ideal class group of $K$, which is finite.  

We shall explain a proposition on the integral points on a line, which can be used to construct plenty of integral points on varieties. Given a straight line $L$ over $K$ in the projective space $\bP^n_{K}$ intersecting a subvariety $D$ in two $K$-rational points, it is clear that the $S$-integral points on the complement $\bP^n_{K}\setminus{D}$ lying on $L$ are potentially dense in $L$, because $L\setminus{(L\cap D)}$ is isomorphic to $\mathbb{G}_m$ over $K$. In order to produce a Zariski dense set of integral points, it is more useful if we can take $S$ sufficiently small. The following proposition proposes such $S$. Note that this is a simplified version of \cite[Lemma 3.2.1]{CoZu23}

\begin{prop} \label{prop:Beukers_Lemma}
Let $X\subset \bP^{n}_K$ be a variety over $K$, let $L\subset X$ be a line over $K$, and let $D\subset X$ be a proper closed subvariety of $X$ over $K$. Suppose that the intersection $L\cap D$ consists of two $K$-rational points $A$ and $B$ which are $S$-coprime, and that $L_v\cap D_v = \set{A_v, B_v}$ for all $v\notin S$. Suppose also that $\cO_{S}^{\ast}$ is infinite and $\cO_S$ is a principal ideal domain.  Then the set $L(K)\cap (X\setminus{D})(\cO_S)$ is infinite. 
\end{prop}

\begin{proof}
We can write $A = [a_0:\dots : a_n]$ and $B = [b_0:\dots: b_n]$ with $a_i, b_i \in \cO_S$ for $i=0,\dots, n$. Since $\cO_S$ is a principal ideal domain, we may suppose that 
\[(a_0,\dots, a_n)\cO_S = (b_0,\dots, b_n)\cO_S = \cO_S\]
as ideals of $\cO_S$. 
 
Let $u\in \cO_S^{\ast}$. Let us show that the point
\[
P(u)\coloneqq [ua_0 + b_0 : ua_1 + b_1 : \dots : ua_n + b_n]
\]
does not reduce to $A$ and $B$ outside $S$. By \cref{prop:Coprime-vs-ideals}, we must show 
\begin{align*} 
&\sum_{i,j}\left\{ (ua_i + b_i)a_j - (ua_j + b_j)a_i \right\} \cO_S = (ua_0 + b_0,\dots, ua_n + b_n)(a_0, \dots, a_n) \\ 
& \sum_{i,j}\left\{ (ua_i + b_i)b_j - (ua_j + b_j)b_i \right\} \cO_S = (ua_0 + b_0,\dots, ua_n + b_n)(b_0, \dots, b_n).
\end{align*} 
By assumption, $(a_0,\dots, a_n)$ and $(b_0,\dots, b_n)$ are the unit ideal of $\cO_S$. Furthermore, the left hand sides are both $\sum_{i,j} (a_ib_j - a_jb_i)$ and these are the unit ideal of $\cO_S$ by \cref{prop:Coprime-vs-ideals}. So, we must show $I\coloneqq (ua_0 + b_0, \dots, ua_n + b_n)$ is the unit ideal. Suppose that $\maf{m}_v$ divides $I$ for some $v\notin S$. Then $(-u)a_i \equiv b_i \pmod{\maf{m}_v}$ and the reduction of $A$ at $v$ is 
\begin{align*}
A_v &= [a_0 \bmod{\mathfrak{m}_v} : \dots: a_n \bmod{\mathfrak{m}_v}] = [(-u)a_0 \bmod{\mathfrak{m}_v} : \dots: (-u)a_n \bmod{\mathfrak{m}_v}], 
\end{align*}
which is exactly $B_v = [b_0\bmod{\mathfrak{m}_v} : \dots: b_n \bmod{\mathfrak{m}_v}]$. This contradicts with the assumption that $A$ and $B$ are $S$-coprime. Therefore, the family $\set{P(u)}_{u\in \cO_S^{\ast}}$ gives infinitely many $S$-integral points on $L\setminus{\set{A,B}}$.
\end{proof}

\begin{rem}
For our purposes, we have considered only the case of lines in the projective space. 
The assumption that $\cO_S$ is a principal ideal domain may be weaken, and the same may be true for general smooth rational curves over $K$ in $\bP^n_K$ according to \cite{CoZu23}.
\end{rem}

As above, rational curves with one $K$-rational point or with two $S$-coprime points at infinity which does not reduce to any curve on the divisor $D$ are called \emph{fully integral curves} (\cite[Section 3.2.4]{CoZu23}). 
The significance of this proposition is that we may obtain integral points lying on a straight line without excessive enlargement of $S$ or $K$. This implies that the Zariski closure $\overline{(\bP^{n}_{K}\setminus{D})(\cO_S)}$ contains the straight line. So the potential density of integral points on varieties may be acquired if there are sufficiently many fully integral curves. This is the key to prove our \cref{MainThm1_n-1planes_plus_1quadric}. See also \cite[Lemma 25]{LY16} for a rigorous statement.

\begin{rem} 
The potential density of fully integral curves {\it with at least one $S$-integer point} is first proved by Beukers when they are embedded in $\bP^2_{K}$ (see \cite[Theorem 2.3]{Beu95}). Hassett and Tschinkel generalized to general rational curves (see \cite[Section 5.2]{HT01}). In the main theorem of \cite{ABP09}, it gives a characterization for (affine) rational curves over $K$ with infinitely many $S$-integral points.
\end{rem}


\subsection{A remark on application of Beukers' Lemma}
We should keep in mind the difference of the two sets $(L\setminus{(L\cap D)})(\cO_S)$ and $L(K)\cap (X\setminus{D})(\cO_S)$. In other words, we should note that even if a $K$-rational point on a line $L$ does not reduce to $L\cap D = \set{A,B}$, it may reduce to $D\setminus{\set{A,B}}$ and may obtain no integral point on $X\setminus{D}$. An example is as follows.   
\begin{ex}
Let $K=\mb{Q}(\sqrt{5})$, $S=M_{K}^{\infty}$. Note that $\cO_S^{\ast} = \cO_K^{\ast} = \set{\pm \epsilon^n \mid n\in \mathbb{Z}}$, where $\epsilon = (-1+\sqrt{5})/2$. 
Let $X=\bP^2_{K}$, and let
\begin{align*} 
D\coloneqq \{ [0:1:0], [1:0:0], [1:1:13], [\epsilon:1:13], [\epsilon^2:1:13],\dots ,[\epsilon^{27}:1:13] \}.  
\end{align*}
Let $L$ be the line $V_+(X_2)$. Then $L\cap D$ consists of two coprime $K$-rational points $[0:0:1]$ and $[1:0:0]$. Let $w$ be the 13-adic valuation on $K$. For any $v\notin S\setminus{\set{w}}$, we have 
\[L_v \cap D_v = \{[0:0:1], [1:0:0]\},\] 
and $L_{w}\cap D_{w}$ consists of 30 $\mathbb{F}_{169}$-rational points. Hence we have $L_v \not\subset D_v$ for all $v\notin S$. However, the set $L(K)\cap (X\setminus{D})(\cO_K)$ is empty. Indeed, since we have
\[
\epsilon^{14} = (-13\epsilon + 8)^2 \equiv -1,\quad \epsilon^{28}\equiv 1\pmod{13},
\]
any integral points $[u:1:0]\in (L\setminus{(L\cap D)})(\cO_K)$ reduces to $D$ at $w$ for all $u\in \cO_K^{\ast}$. 
\end{ex}
In contrast to \cref{prop:Beukers_Lemma}, the condition $L_v \cap D_v = \set{A_v, B_v}$ does not appear in \cite[Lemma 3.2.1 (b)]{CoZu23}. It is because when $D$ is a divisor, we can confirm the condition $L_v \cap D_v = \set{A_v, B_v}$ from $L_v\not\subset D_v$. For example, let $D$ be a quadric curve in $\bP^{2}_K$ and $L$ be a line in $\bP^{2}_K$. Suppose that the intersection $L\cap D$ consists of two $S$-coprime $K$-rational points. Then, if we know that $L_v \cap D_v$ is zero-dimensional, we have $L_v \cap D_v = \set{A_v, B_v}$, because $L_v\cap D_v$ already contains two points and because a line and a quadric curve usually intersect in two points. 
We summarize this observation as follows. 
 
\begin{prop} \label{prop:Beukers_lem_divisor}
Let $L\subset \bP^{n}_{K}$ be a line over $K$, and let $D\subset \bP^{n}_K$ be a divisor over $K$. We suppose the following. 
\begin{itemize}
\item The intersection $L\cap D$ consists of two $K$-rational points $A,B$. 
\item The reduction $L_v$ is not contained in $D_v$ for all $v\notin S$. 
\item The points $A$ and $B$ are $S$-coprime. 
\item The $S$-unit group $\cO_{S}^{\ast}$ is infinite and that $\cO_S$ is a principal ideal domain.
\end{itemize}
Then the set $L(K)\cap (X\setminus{D})(\cO_S)$ is infinite.
\end{prop}
\begin{proof} Since $D$ is a divisor, it is defined by a degree $n$ homogeneous polynomial over $K$. Let $m_A$ and $m_B$ be the multiplicity of the intersection at $A$ and $B$, respectively. Then $L_v$ and $D_v$ intersect at two distinct $k_v$-rational points $A_v$ and $B_v$ with multiplicity $\geq m_A$ and $\geq m_B$. Since $m_A + m_B = n$, both equality hold. This shows $L_v\cap D_v = \set{A_v, B_v}$, and the proof is complete by \cref{prop:Beukers_Lemma}. 
\end{proof}
\section{Main Results} \label{section:main_result}
In this section, we prove the main results of this paper.

\subsection{Integral points on a quadric hypersurface} 
In this section, we study the integral points on $Q$ with respect to hyperplane sections. The proposition in this section is a generalization of \cite[Lemma 3.3.1, Proposition 4.2.1]{CoZu23}, and is applied in \cref{MainThm1_n-1planes_plus_1quadric}. 

Let us recall assumptions in \cref{MainThm1_n-1planes_plus_1quadric}. 

\begin{ass} \label{assumption}
Let $n\geq 2$, and let $D=H_1 + \dots + H_{n-1} + Q$, where $Q$ is a quadric hypersurface in $\bP_K^{n}$ and $H_1,\dots, H_{n-1} \subset \bP^{n}_{K}$ are hyperplanes in general position, all defined over $K$. Let $L \coloneqq H_1\cap \dots \cap H_{n-1}$. We assume that $Q$ and $L$ have two $K$-rational intersection point. Let $p$ be one of them, and assume that $Q$ is smooth at $p$. Denote $T_{p}Q$ as the tangent hyperplane at $p$. 
\end{ass}

\begin{prop} \label{intpt_on_Q}
Let $E = H_1 + \dots + H_{n-1}$. Then the integral points on $Q\setminus{(\Supp(T_{p}Q + E)\cap Q)}$ are potentially dense. 
\end{prop}

\begin{proof}
We may assume that the divisor $T_{p}Q + E$ is reduced.  

By drawing the line $l_q$ joining $p$ and $q\in \bP^n_K(K)\setminus{\set{p}}$, we obtain a $K$-morphism $\pi: \bP^{n}_{K}\setminus{\set{p}} \to \bP_K^{n-1}$. 
The line $l_q$ intersects $Q\setminus{p}$ in another point if and only if $l_q$ is not contained in $T_{p}Q$. So we obtain an isomorphism over $K$
\[
\pi': Q\setminus{\set{p}} \xrightarrow{\cong} \bP^{n-1}_{K} \setminus{\pi(T_{p}Q)}.
\]
Let $\Pi_T\coloneqq \pi(T_{p}Q)$ and let $\Pi_i \coloneqq \pi(H_i)$ for $1\leq i\leq n-1$. Then $\Pi_i$ and $\Pi_T$ are hyperplanes in $\bP^{n-1}_K$. The morphism $\pi'$ induces 
\[
Q\setminus{(\Supp(T_{p}Q + E)\cap Q)} \xrightarrow{\cong} \bP^{n-1}_{K} \setminus{(\Pi_1 + \dots + \Pi_{n-1} + \Pi_T)}.
\]
 Since the line $L=H_1\cap\dots\cap H_{n-1}$ intersects $Q$ in two $K$-rational points, we have $L\not\subset T_{p}Q$. Thus, the hyperplanes $H_1,\dots, H_{n-1}$ and $T_{p}Q$ does not contain any common lines, or equivalently, the hyperplanes $\Pi_1,\dots, \Pi_{n-1}$ and $\Pi_T$ are in general position. Therefore, we have an isomorphism
\[\bP^{n-1}_{K}\setminus{\Supp(\Pi_1+\dots + \Pi_{n-1} + \Pi_T)} \cong \mb{G}_{m}^{n-1}\]
over $K$. Then the integral points on both sides are potentially dense, and those on $Q\setminus{(\Supp(T_pQ + E)\cap Q)}$ are also potentially dense. 
\end{proof}

\subsection{Proof of \cref{MainThm1_n-1planes_plus_1quadric}} 
Applying \cref{intpt_on_Q}, now we can prove \cref{MainThm1_n-1planes_plus_1quadric}.
\begin{proof}[Proof of \cref{MainThm1_n-1planes_plus_1quadric}] Let $U \coloneqq Q\setminus{(\Supp{(T_{p}Q + E)}\cap Q})$. 
Enlarging $S$, we may assume that $\cO_{S}$ is a principal ideal domain, that $\cO_S^{\ast}$ is infinite, that $(T_{p}Q)_v\cap Q_v = (T_{p}Q \cap Q)_v$ and $E_v \cap Q_v = (E\cap Q)_v$ for all $v\notin S$ by \cref{lem:bad-reduction-intersection}, and that the integral points on $U$ is Zariski dense in $Q$ by \cref{intpt_on_Q}. 

Let $q\in U(\cO_S)$ be an integral point, and let $l_q$ be the line joining $p$ and $q$. Let us show that the line $l_q$ is fully integral. By \cref{prop:Beukers_lem_divisor}, it is sufficient to show that the reduction $(l_q)_v$ is not contained in $D_v$ for all $v\notin S$. 
Indeed, the point $q_v \in Q_v$ is not contained in $E_v\cap Q_v$ and $(T_{p}Q)_v \cap Q_v$ for all $v\notin S$ by the definition of $q$, and it follows that we have $(l_q)_v\not\subset E_v\cup (T_{p}Q)_v$ and hence $(l_q)_v\not\subset Q_v$. Therefore, it follows that the line $l_q$ is fully integral and the set $l_q(K)\cap (\bP^{n}\setminus{D})(\cO_S)$ is infinite.  

Now we have following:
\begin{align*}
    Z&\coloneqq \overline{\bigcup_{q\in U(\cO_S)} l_q(K)} \\
&= \overline{\bigcup_{q\in U(\cO_S)} l_q(K)\cap (\bP^{n}\setminus{D})(\cO_S)} \\
&\subset \overline{(\bP_K^n\setminus{D})(\mathcal{O}_S)}
\end{align*}

 So it is sufficient to prove that $Z = \bP_K^n$. Let $H$ be any hyperplane over $K$ not containing $p$. By considering a projection from $p$, we obtain a dominant morphism $f: Q\setminus{\set{p}} \to H$. The image $f(U(\cO_S))$ is contained in $Z$ and Zariski dense in $H$. Therefore, $Z$ contains $\bigcup_{H\not\ni p}H$. So we obtain $Z=\bP^{n}_{K}$. 
\end{proof}
By \cref{MainThm1_n-1planes_plus_1quadric}, we immediately obtain a generalization of \cite[Theorem 3.3.2]{CoZu23}. 

\subsection{Proof of \cref{MainThm3_concurline_puncturing}} \label{sect:MainThm3}
Let us prove our last theorem. 
\begin{proof}[Proof of \cref{MainThm3_concurline_puncturing}] \, \\
\emph{Step 1. (Coordinate change)}

Changing the coordinates, we may suppose 
\[H_1 = V_+(X_0), \quad H_2 = V_+(X_1),\quad H_3 = V_+ (X_2), \quad H_4 = V_+(X_3).\]
Then, we have
    \begin{multline*}
    (\bP^{3}_K\setminus{D})(\cO_S) = \\ 
    \set{[\alpha:\beta:\gamma : 1] \in \bP^{3}_K(K) \mid \alpha,\beta,\gamma \in \cO_S^{\ast}, \quad (\alpha,\beta,\gamma)_v \notin (L_1\cup\dots \cup L_r)_v \quad \text{for all  $v\notin S$} }. 
    \end{multline*}
Let us write $C_{ij} \coloneqq H_i\cap H_j$ for distinct $i,j\in \set{1,2,3,4}$. We may suppose that the intersection point $p$ of the lines $L_1,\dots, L_r$ is in the affine open set $\mathbb{A}^{3}_K = \bP^{3}\setminus{V_+(X_3)}$, so let us write 
\[p\coloneqq[b:d:f:1]\] 
for some $b,d,f\in K$. Since $L_1,\dots,L_r$ do not intersect the three lines $C_{12}\cup C_{23}\cup C_{13}$, it follows that $b,d,f$ are not zero. The line $L_i$ is written as 
\[
L_i(K) = \set{[a_i t + bs : c_i t + ds : e_it + fs : s] \mid [t:s]\in \bP^1_K(K)}
\]
for some $a_i, c_i, e_i\in K$. Note that $a_i,c_i,e_i$ are not zero because $L_i$ does not intersect the three lines $C_{14}\cup C_{24}\cup C_{34}$. 

The subvariety $\bigcup_{i=1}^{r}L_i$ contains $(x_0,x_1,x_2) \in \mathbb{A}^{3}_K(K)\setminus{\set{p}}$ if and only if the following holds:
\[
[x_0 - b: x_1 - d: x_2 - f] \in \set{[a_i : c_i : e_i] \in \bP^{2}_{K}(K) \mid i=1,2,\dots, r}.
\]
So, the condition $(\alpha,\beta,\gamma)_v \notin L_v$ is equivalent to say that the point $[\alpha - b: \beta - d: \gamma - f]$ does not reduce to the right hand side. 

\emph{Step 2. (Construction of integral points)}

Enlarging $S$ if necessary, we may suppose that $a_i,b,c_i,d,e_i,f\in \cO_S^{\ast}$ for all $i\in \set{1,\dots, r}$ and that the unit group $\cO_S^{\ast}$ admits elements $\alpha,\beta,\gamma$ of infinite order. 
For integers $j,k$ and $i\in \set{1,2,\dots, r}$, let $I_{i,j,k}$ and $I_{j,k}$ be the ideals of $\cO_S$ given by 
\begin{align*}
I_{i,j,k} &\coloneqq 
    (e_i(\beta^{j} - d) - c_i(\gamma^{k} - f_i) )\cO_S \\
I_{j,k} &\coloneqq \prod_{i=1}^{r} I_{i,j,k}
\end{align*}
 When the pair $(j,k)$ satisfies $I_{j,k}\neq (0)$ and $v$ is a place outside $S$, we define the integers $g_{v,j,k}, g_{j,k}, N_{j,k}$ as follows: 
\begin{align*}
g_{v,j,k} &\coloneqq \min{\set{v(\beta^j - d), v(\gamma^k - f)}},\\
g_{j,k}&\coloneqq \max_{\maf{m}_v \mid I_{j,k}}{g_{v,j,k}},\\
N_{j,k} &\coloneqq \#((\cO_S/I_{j,k}^{1+g_{j,k}})^{\ast}).
\end{align*}
Note that $I_{j,k}\neq (0)$ and $\beta,d,\gamma,f\in \cO_S^{\ast}$ implies that these numbers are finite and non-negative. 
Now, we construct $x_{j,k,l}$ for $l\in \mathbb{Z}$ as follows: 
\[
x_{j,k,l}\coloneqq (b\alpha^{lN_{j,k}},\, \beta^{j},\, \gamma^{k}).
\]
Note that since $b,\alpha,\beta$ and $\gamma$ are $S$-unit, it follows that $x_{j,k,l}$ does not reduce to the four planes $\bigcup_{i=1}^{4}H_i$. 

\emph{Step 3.(Integrality of $x_{j,k,l}$)}

Let us show that $x_{j,k,l}\in (\bP^{3}_{K}\setminus{D})(\cO_S)$. By \cref{prop:Coprime-vs-ideals} and Step 1., it is sufficient to show the following equation of ideals of the valuation ring $\cO_v$ for all $v\notin S$ and $i\in \set{1,\dots, r}$: 
\begin{multline} \label{ideal-equation-4.1}
(b\alpha^{lN_{j,k}} - b,\, \beta^{j} - d,\, \gamma^{k}-f)\cO_v \\ 
= (c_i(b\alpha^{lN_{j,k}} - b) - a_i(\beta^{j} - d),\, e_i(\beta^{j} - d) - c_i(\gamma^{k} - f),\, e_i(b\alpha^{lN_{j,k}} - b) - a_i(\gamma^{k} - f))\cO_v. 
\end{multline}

We also denote $\mathfrak{m}_v$ by the maximal ideal of $\cO_S$ corresponding to $v$. If $\maf{m}_v$ does not divide $I_{i,j,k}$, then the both hand sides of \cref{ideal-equation-4.1} contain $e_i(\beta^j - d) - c_i(\gamma^k - f)$, which is a unit in $\cO_v$. This implies that they are unit ideal. If $\maf{m}_{v}$ divides $I_{i,j,k}$, then we have 
\[I_{j,k}^{1+g_{j,k}} \subset I_{i,j,k}^{1+g_{j,k}} \subset \maf{m}_v^{1 + g_{v,j,k}}\]
and hence
\[
b\alpha^{lN_{j,k}} - b \equiv b\cdot 1^{l} - b \equiv 0\pmod{\maf{m}_v^{1 + g_{v,j,k}}}.
\]
This implies $v(b\alpha^{lN_{j,k}} - b) > g_{v,j,k}$. Combining this inequality with $a_i \in \cO_S^{\ast}$, we obtain
\begin{align*}
g_{v,j,k} &= \min{\set{v(\beta^j - d), v(\gamma^k - f)}}\\ 
&=\min{\set{v(c_i(b\alpha^{lN_{j,k}} - b)- a_i(\beta^j - d)),\, v(e_i(b\alpha^{lN_{j,k}} - b) - a_i(\gamma^{k} - f))}}.
\end{align*}
Note also that $v(e_i(\beta^j - d) - c_i(\gamma^k - f)) \geq g_{v,j,k}$ by the definition of $g_{v,j,k}$.
Therefore, the both hand sides of \cref{ideal-equation-4.1} are exactly $\maf{m}_{v}^{g_{v,j,k}}\cO_v$. Hence the point $x_{j,k,l}$ does not reduce to $L_i$ for all $i$, and is an $S$-integral point on $\bP^{3}_{K}\setminus{D}$ for a nice pair $(j, k)$ and $l\in \mathbb{N}$. 

\emph{Step 4. (Potential density of integral points)}

Let us fix $k\in \mathbb{Z}$. Since $\beta$ is of infinite order in $\cO_S^{\ast}$ and since $e_i \neq 0$, all of $j\in \mathbb{Z}$ except for at most $r$ integers satisfies $e_i(\beta^{j} - d) - c_i(\gamma^{k} - f) \neq 0$ for any $i$.
This implies $I_{j,k}\neq (0)$, and the point $x_{j,k,l}$ is an $S$-integral point on $\bP^{3}_{K}\setminus{D}$ for all $l\in \mathbb{N}$. Since $\alpha, \beta$ is also of infinite order and $b\neq 0$, the set of all the points $x_{j,k,l}$ for such $j, l$ is Zariski dense in the hyperplane $V_+(X_2 - \gamma^k X_3)$. 
Therefore, it follows that the set $(\bP^{3}_{K}\setminus{D})(\cO_S)$ contains infinitely many hyperplanes $\bigcup_{k\in \mathbb{Z}} V_+(X_2 - \gamma^k X_3)$, and the desired potential density is obtained. 
\end{proof}

Although not contained in \cref{MainThm3_concurline_puncturing}, the following case also may be proven:
\begin{prop}
Let $D, H_i, L_i, p$ be as in \cref{MainThm3_concurline_puncturing}. Suppose that $p\in H_i\cap H_j\cap H_k$ for some $1\leq i<j<k\leq 4$. Then the integral points on $\bP^{3}_{K}\setminus{D}$ are potentially dense. 
\end{prop}
\begin{proof}
 Let $(i,j,k) = (1,2,3)$ for instance, and let $L_i \cap H_4 = \set{x_i}$. Then the integral points on 
 \[ H_4\setminus{(D\cap H_4)} = H_4\setminus{((\Supp(H_1+H_2+H_3)\cap H_4) \cup \set{x_1,\dots, x_r})}
 \]
 are potentially dense by \cref{threelines_points}. Enlarging the set $S$, we may suppose the following: 
 \begin{itemize}
 \item $\cO_S$ is a principal ideal domain and the unit group $\cO_S^{\ast}$ is infinite. 
 \item The set $(H_4\setminus{(H_4\cap D)})(\cO_S)$ is Zariski dense in $H_4$. 
 \item $(L_i)_v \cap (H_4)_v = (L_i \cap H_4)_v =: \set{(x_i)_v}$ and $p_v \notin (H_4)_v$. (By \cref{lem:bad-reduction-intersection})
 \end{itemize} 
 Let $q\in (H_4\setminus{(H_4\cap D)})(\cO_S)$ be an integral point, and let $l_q$ be the line connecting $p$ and $q$. If $q_v \in (L_i)_v$, then we have $q_v \in (H_4)_v \cap (L_i)_v = \set{(x_i)_v}$, a contradition. 
 Therefore, it follows that $(l_q)_v\cap D_v = \set{p_v, q_v}$. Since $p_v$ is different from $p_v\in (H_4)_v$, the straight line $l_q$ is fully integral. So we easily find that the integral points on $\bP^{3}_{K}\setminus{D}$ are potentially dense by applying \cref{prop:Beukers_Lemma} to the line $l_q$.
\end{proof}
Note that the coordinate of the point $x_{j,k,l}$ we constructed in \cref{MainThm3_concurline_puncturing} refers to $b\in K$, which is the first coordinate of the intersection point $p$ of $L_1,\dots,L_r$. So, without the concurrency, it seems to be difficult to construct a point in the same way as $x_{j,k,l}$. 

\subsection{On 3-dimensional non-normal crossings cases}
The condition $\# (L\cap Q) = 2$ is necessary for our proof of \cref{MainThm1_n-1planes_plus_1quadric}, because it implies that the hyperplanes $\Pi_1, \dots ,\Pi_{n-1}, \Pi_T$ are in general position. Here we work on $\bP^{3}$, and we remove the condition “normal crossings” of $D = H_1 + H_2 + Q$. 

Let $L\coloneqq H_1\cap H_2$ be the line. We consider two cases below: 
\begin{enumerate}
\item $L\subset Q$.
\item $L \subset T_{p}Q$ and $L\not\subset Q$. 
 \end{enumerate}

For (1), we prove that the integral points on $\bP^{3}\setminus{D}$ remains to be potentially dense. The method is the same (and even easier) as \cref{MainThm1_n-1planes_plus_1quadric} and becomes even easier. 
\begin{prop} 
Let $D$ be a divisor on $\bP^{3}_K$ of the form $D \coloneq H_1 + H_2 + Q.$
Here, $H_1, H_2$ are two distinct hyperplanes and $Q$ is a smooth quadric hypersurfaces, all defined over $K$. Suppose that the line $L\coloneqq H_1\cap H_2$ is contained in $Q$. 
Then the integral points of $\bP^3\setminus{D}$ are potentially dense. 
\end{prop}
\begin{proof}
Let us write $Q\cap H_i = L \cup L_i$, where $L_i$ is a straight line distinct from $L$. 
Let $p_i$ ($i=1,2$) be the point at which the lines intersect. Note that $H_i = T_{p_i}Q$. The projection from $p_1$ induces isomorphisms $Q\setminus{H_1} \xrightarrow{\cong} \bP^{2}\setminus{\Pi_1}$ and 
\[
Q\setminus{(H_1\cup H_2)} \cong \bP^{2}\setminus{(\Pi_1\cup \Pi_2)} \cong \mathbb{A}^1\times \mathbb{G}_m. 
\]
Thus $Q\setminus{(H_1\cup H_2)}$ has potentially dense set of $S$-integral points. For any $S$-integral points $q$ on the complement,  let $l_q$ be the line passing through $q$ and $p_1$. Since $(l_q)_v \not\subset (H_1\cup H_2 \cup Q)_v$ for all $v\notin S$, we have $(l_q\cap D)_v = \set{p_{1v}, q_v} = (l_q)_v \cap D_v$ and $q, p_{1v}$ are $S$-coprime. By \cref{prop:Beukers_lem_divisor}, the line $l_q$ has infinitely many $S$-integral points of $Q\setminus{(H_1\cup H_2)}$. The same argument in \cref{MainThm1_n-1planes_plus_1quadric} shows that $(\bP^{3}\setminus{D})(\cO_S)$ is Zariski dense, after enlarging $S$ so that $(Q\setminus{(H_1\cup H_2)})(\cO_S)$ is Zariski dense. 
\end{proof}

\begin{ex}
Let $H_1 = V_+(X_1)$, $H_2=V_+(X_2)$ and $Q = V_+(X_0X_1 + X_2X_3)$ be smooth hypersurfaces in $\mathbb{P}^{3}_K$. Then the intersection $H_1 \cap H_2\cap Q = V_+(X_1)\cap V_+(X_2)$ is a line. So the integral points on $\bP^{3}_{K}\setminus{\Supp (H_1+H_2+Q))}$ is potentially dense. 
\end{ex}

For (2), we have not determined whether integral points are potentially dense or not. The projection from an intersection point of components can not produce a Zariski dense set of integral points, because the corresponding lines $\Pi_1,\Pi_2,\Pi_T$ (see \cref{intpt_on_Q}) are concurrent. So we propose the following.
\begin{prob}
Let $D$ be a divisor on $\bP^{3}_K$ of the form $D \coloneq H_1 + H_2 + Q.$
Here, $H_1$ and $H_2$ are two distinct hyperplanes and $Q$ is a smooth quadric hypersurfaces, all defined over $K$. Suppose that the line $L\coloneqq H_1\cap H_2$ tangent to $Q$ at a $K$-rational point. 
Is the integral points of $\bP^3\setminus{D}$ are not potentially dense?  
\end{prob}
Since the divosor $D$ above is not normal crossing divisor, the computation of logarithmic kodaira dimension $\overline{\kappa}(\bP^{3}\setminus{D})$ becomes more subtle. Like as Levin-Yasufuku's work \cite{LY16}, a (3-dimensional) characterization of complements which connects potential density to the value $\overline{\kappa}$ is awaited.

\subsection*{Acknowledgements}
The author would like to thank his advisor, Tetsushi  Ito, for helping to progress my research and writing, and Yu Yasufuku for useful discussions and warm encouragement.

\bibliography{teranishi}

\end{document}